\documentclass[preprint,3p,times,12pt]{elsarticle}
\usepackage{amsmath}
\usepackage{diagbox}
\usepackage{mathrsfs}
\usepackage{tikz}
\usepackage{multirow}
\usepackage{slashbox}
\usepackage{hyperref}
\usepackage{listings}
\usepackage{multirow}
\usepackage{amssymb}
\usepackage{amsthm}
\usepackage{epstopdf}
\usepackage{algorithm,float}
\usepackage{algpseudocode}
\usepackage{subfigure}
\usepackage{tabularx,booktabs}
\usepackage{caption}
\usepackage{verbatim}
\usepackage[figuresright]{rotating}
\usepackage{CJK}
\usepackage[all]{xy}
\renewcommand{\thealgorithm}{\arabic{algorithm}}
\biboptions{sort&compress}
\theoremstyle{plain}
\newtheorem{theorem}{Theorem}[section]
\newtheorem{lemma}[theorem]{Lemma}
\newtheorem{proposition}[theorem]{Proposition}
\theoremstyle{definition}
\newtheorem{definition}[theorem]{Definition}
\newtheorem{remark}[theorem]{Remark}
\newtheorem{corollary}[theorem]{Corollary}
\newtheorem{example}[theorem]{Example}

\journal{xxxx}

\makeatletter
\newif\if@borderstar
\def\bordermatrix{\@ifnextchar*{%
\@borderstartrue\@bordermatrix@i}{\@borderstarfalse\@bordermatrix@i*}%
}
\def\@bordermatrix@i*{\@ifnextchar[{\@bordermatrix@ii}{\@bordermatrix@ii[()]}}
\def\@bordermatrix@ii[#1]#2{%
\begingroup
\m@th\@tempdima8.75\p@\setbox\z@\vbox{%
\def\cr{\crcr\noalign{\kern 2\p@\global\let\cr\endline }}%
\ialign {$##$\hfil\kern 2\p@\kern\@tempdima &\thinspace %
\hfil $##$\hfil &&\quad\hfil $##$\hfil\crcr\omit\strut %
\hfil\crcr\noalign{\kern -\baselineskip}#2\crcr\omit %
\strut\cr}}%
\setbox\tw@\vbox{\unvcopy\z@\global\setbox\@ne\lastbox}%
\setbox\tw@\hbox{\unhbox\@ne\unskip\global\setbox\@ne\lastbox}%
\setbox\tw@\hbox{%
$\kern\wd\@ne\kern -\@tempdima\left\@firstoftwo#1%
\if@borderstar\kern2pt\else\kern -\wd\@ne\fi%
\global\setbox\@ne\vbox{\box\@ne\if@borderstar\else\kern 2\p@\fi}%
\vcenter{\if@borderstar\else\kern -\ht\@ne\fi%
\unvbox\z@\kern-\if@borderstar2\fi\baselineskip}%
\if@borderstar\kern-2\@tempdima\kern2\p@\else\,\fi\right\@secondoftwo#1$%
}\null \;\vbox{\kern\ht\@ne\box\tw@}%
\endgroup
}
\makeatother
\allowdisplaybreaks

\allowdisplaybreaks

\makeatletter
\def\ps@pprintTitle{%
	\let\@oddhead\@empty
	\let\@evenhead\@empty
	\let\@oddfoot\@empty
	\let\@evenfoot\@empty
}
\makeatother

\begin{document}
\begin{frontmatter}

\title{\textbf{$PF$-well-filtered spaces}}
\author{\textbf{Ruoxi Chen}\textsuperscript{a}}
\ead{chenruoxi@hnu.edu.cn}                 
\author{\textbf{Qingguo Li}\textsuperscript{a}}   
\ead{liqingguoli@aliyun.com}  
\author{\textbf{Longchun Wang}\corref{cor1}\textsuperscript{b}}
\ead{longchunw@163.com}

\address{\textsuperscript{a}School of Mathematics, Hunan University, Changsha, Hunan, 410082, \textit{P. R. China}}
\address{\textsuperscript{b}School of Mathematical Sciences, Qufu Normal University, Qufu, Shandong 273165, China}
\cortext[cor1]{Corresponding author.}

\begin{abstract}
	This paper solves two open problems concerning the well-filteredness of topological spaces. The first, due to Shen, Xi and Zhao\cite{shen2023further}, asks whether a space must be open well-filtered whenever its upper space is so; the answer is negative; the Hofmann-Lawson space provides a counterexample. The second, posed by Li, Yuan and Zhao\cite{li2020unified}, inquires whether $PF$-well-filteredness and $PF$-sobriety are equivalent under the assumption of local compactness; we answer it in the affirmative. Subsequently, we give characterizations of $PF$-sobriety and $PF$-filteredness for $T_0$ spaces, respectively. On the basis of the characterizations, we obtain the following result: $PF$-sober and $PF$-well-filter coincide on Rudin spaces, which extends the previously established equivalence for locally compact spaces, as every locally compact space is a Rudin space. Moreover, we also explore function spaces involving the $PF$-well-filtered property, together with the reflective hull of the category of $PF$-well-filtered spaces in $T_0$ spaces.

\end{abstract}


\end{frontmatter}

\section{Introduction}

Sobriety, well-filteredness and $d$-spaces are fundamental notions in non-Hausdorff topology and domain theory\cite{gierz2003continuous}\cite{goubault2013non}. They describe different aspects of the interaction between topology and the specialization order. These properties satisfy
\[
\text{sobriety}\ \Longrightarrow\ \text{well-filteredness}
\Longrightarrow\ d\text{-space}
\]
and neither implication is reversible in general.

Well-filtered spaces are a class of topological spaces introduced by Heckmann that have been extensively studied to date\cite{gierz2003continuous}. The property of well-filteredness guarantees that all nonempty compact saturated sets in the space form a dcpo under the reverse inclusion order, making them particularly important in logical and denotational semantics\cite{goubault2013non}. Several related notions have been introduced by modifying the families or the open sets involved in this condition. These include weak well-filteredness\cite{lu2017weak}, open well-filteredness\cite{shen2020open}, strong well-filteredness\cite{xu2025strongly} and $PF$-well-filteredness\cite{li2020unified}. Studying these variants helps identify which parts of the theory depend on the full well-filteredness condition and which remain valid under weaker assumptions.

Open well-filteredness, introduced by Shen, Xi, Xu and Zhao\cite{shen2020open} replaces filtered families of compact saturated sets with families of open sets filtered with respect to the way-below relation on \(\mathcal O(X)\). They proved that every well-filtered space is open well-filtered and every core-compact open well-filtered space is sober, which solved the Jia-Jung problem. They subsequently studied these spaces. In 2023, Shen, Xi and Zhao proved that  open well-filteredness is preserved by the Smyth powerspace and they asked "Is it true that a space is open well-filtered if its upper space is open well-filtered?"\cite{shen2023further}. In Section 3 of this paper, we give a negative answer to this problem by using the Hofmann-Lawson space.

Another line of research concerns a new variant of well-filteredness. In \cite{li2020unified}, Li, Yuan and Zhao introduced the notion of $PF$-well-filteredness by modifying the definition of well-filteredness: the intersection of a filtered family of nonempty compact saturated subsets is required to be a principal filter $\uparrow x$ (with respect to the specialization order) contained in an open set U. In an analogous manner, they also introduced the notion of $PF$-sobriety. They proved that $PF$-sobriety implies $PF$-well-filteredness, and posed the question of whether the two properties are equivalent for locally compact spaces. In Section 3 of this paper, we provide an affirmative answer to this problem and establish characterizations of $PF$-sobriety and $PF$-well-filteredness for $T_0$ spaces, respectively. On the basis of the characterizations, we obtain the following result: $PF$-sobriety and $PF$-well-filteredness coincide on the Rudin spaces, which extends the previously established equivalence for locally compact spaces, as every locally compact space is a Rudin space. In Section 4 of this paper, we supply a counterexample to show that for nonempty $T_0$ topological spaces $X$ and $Y$, if $Y$ is a $PF$-well-filtered ($PF$-sober) space, then the function space consisting of all continuous maps from $X$ to $Y$ equipped with the Isbell topology may not be a $PF$-well-filtered ($PF$-sober) space even when $X$ is a core-compact space. We also prove that if the function space composed of all continuous maps from $X$ to $Y$ equipped with the Isbell topology is $PF$-well-filtered ($PF$-sober), then $Y$ is also $PF$-well-filtered ($PF$-sober) even though $PF$-sobriety and $PF$-well-filteredness are not preserved under retracts. Finally, in Section 5 of this paper, we show that the reflective hull of the category of $PF$-well-filtered ($PF$-sober) spaces with continuous maps in $TOP_0$ is $TOP_0$.

\section{Preliminaries}
In this section, we review some basic concepts to be used in this paper. For more details, see the references \cite{gierz2003continuous}.

Let $P$ be a poset and $F \subseteq P$. We denote that $\downarrow F=\{x\in P \mid \exists f\in F, x\le f\}$ and $\uparrow F=\{x\in P\mid \exists f\in F, x\ge f \}$. For each $x\in P$, $\downarrow \{x\}=\downarrow x=\{y\in P\mid y\le x\}$ and $\uparrow \{x\}=\uparrow x=\{y\in P\mid y\ge x\}$. A subset $F$ is called a lower set(resp., an upper set) if $F=\downarrow F$(resp., $F=\uparrow F$). We write $F^u=\{x\in P\mid y\le x \,\text{for each} \,y\in F\}$ and $F^l=\{x\in P\mid y\ge x \,\text{for each} \,y\in F\}$ for the sets of upper and lower bounds of $F$, respectively, and set $F^{ul}=(F^{u})^{l}$. The subset $F$ is called a cut if $F=F^{ul}$.

A nonempty subset $D$ of $P$ is directed if every finite subset of $D$ has an upper bound in $D$. The poset $P$ is called a dcpo if every directed subset $D$ in $P$ has a supremum. A subset $U$ of $P$ is Scott open if (1) $U=\uparrow U$ and (2) for any directed subset $D$ for which $\bigvee D$ exists, $\bigvee D \in U$ implies that $D\cap U\neq \emptyset$. All Scott open subsets of $P$ form a topology, and we call this topology the Scott topology on $P$.

For any topological space $X$, let $F\subseteq X$. $F$ is called irreducible if it is nonempty and $F\subseteq F_1\cup F_2$ implies that $F\subseteq F_1$ or $F\subseteq F_2$ for any $F_1, F_2\in \Gamma(X)$, where $\Gamma(X)$ is the set of all the closed sets of $X$. The set of all irreducible closed sets of $X$ is denoted by $Irr_C(X)$. $F$ is called saturated if it is the intersection of all open sets containing it. The set of all open sets is denoted by $\mathcal{O} (X)$. If $X$ is a $T_0$ space, then the specialization order $\le_X$(in short,$\le$) on $X$ is defined by $x\le y$ if and only if $x\in cl(\{y\})$. Thus the open set of $X$ is an upper set with the specialization order, the closed set of $X$ is a downset. The Smyth power space of a $T_0$ space $X$, denoted by $Q(X)$, which is also called the upper space of $X$, is the set of all nonempty compact saturated sets of $X$ with the upper Vietoris topology. A $T_0$ space is called well-filtered if an open set $U$ contains a filtered intersection of compact saturated sets, then $U$ must already contain one of them. A $T_0$ space is called open well-filtered if an open set $U$ contains a $\ll$-filtered intersection of open sets, then $U$ must already contain one of them. A $T_0$ space is called $PF$-well-filtered if an open set $U$ contains a principal filter $\uparrow x$, which is also a filtered intersection of compact saturated sets, then $U$ must already contain one of them. A $T_0$ space is called $PF$-sober if an open set $U$ contains a principal filter $\uparrow x$, which is also an intersection of an irreducible set of the upper space, then $U$ must already contain one of them. A nonempty subset $A$ of $X$ is said to have the Rudin property if there exists a filtered family $\mathcal{K}\subseteq Q(X)$ such that $\overline{A}$ is a minimal closed set that intersects all members of $\mathcal{K}$. Let $RD(X)=\{A\in\Gamma(X)\mid \text{A has Rudin property}\}$. The sets in $RD(X)$ will also be called the Rudin sets. A $T_0$ space $X$ is called a Rudin space if $Irr_C(X)=RD(X)$\cite{shen2019well}\cite{xu2020t0}.

Now we revisit the Hofmann-Lawson space. By \cite{gierz2003continuous}, Exercise V-5.25, there exists a dense subset $A$ of $[0,1)$ such that $A\cap U$ is not a Borel set for every nonempty open subset $U$ of $[0,1]$. Goubault-Larrecq also elaborated on this in a 2019 blog post. We denote $I=[0,1]$ with the usual topology, let $J=[0,1)$ with the Scott topology under the reverse order and $Y=I\times J$ be the product space. Equip $E=[A\times(J\cap\mathbb{Q})]\cup[(I\setminus A)\times(J\setminus\mathbb{Q})]$ with the relative topology from $Y$, we call $E$ the Hofmann-Lawson space. In \cite{hofmann1978spectral}, Hofmann and Lawson proved that $E$ is a second-countable, core-compact space whose interior of every compact subset is empty, and thus is not a locally compact space. In the remainder of this paper, when we consider the Hofmann-Lawson space, we will use these symbols directly.

\section{Concerning two open problems on well-filteredness}
In this section, we will reveal that the Hofmann-Lawson space $E$ is not an open well-filtered space but its Smyth power space is an open well-filtered space. In addition, we show that $PF$-well-filteredness and $PF$-sobriety coincide on locally compact spaces, thereby answering the question raised in \cite{li2020unified} and give characterizations of  $PF$-well-filtered spaces and of $PF$-sober spaces, respectively. By means of these characterizations, we further extend this equivalence beyond local compactness, namely to the strictly larger class of the Rudin spaces.

\begin{lemma}
	Let $X$ be a topological space in which every compact saturated subset has empty interior. Then the Smyth power space $Q(X)$ is open well-filtered.
\end{lemma}

\begin{proof}
	We show that in the Smyth power space $Q(X)$, for open sets $\mathcal{U}$ and $\mathcal{V}$,$\mathcal{U}\ll\mathcal{V}$ if and only if $\mathcal{U}=\emptyset$.
	
	Let $\mathcal{U}$ be a nonempty open set. Then there exists an open set $W\neq\emptyset$ such that $\emptyset\neq\Box W\subseteq\mathcal{U}$. Consider $\Phi_W=\{\Box O\mid\ O\in\mathcal{O}(X), W\nsubseteq O\}$.
	
	Claim 1: $\Phi_W$ is an open cover of $Q(X)$.
	
	Assume that $K\in Q(X)$, then $W\nsubseteq K$ because $K$ has an empty interior. Take $w\in W\setminus K$. Since $K$ is saturated, we can choose an open set $O$ such that $K\subseteq O$ and $w\notin O$. Thus $K\in\Box O\in\Phi_W$.
	
	Claim 2: There is no finite subfamily of $\Phi_W$ that can cover $\Box W$.
	
	Given $\{\Box O_1,...,\Box O_m\}$ is a finite subfamily of $\Phi_W$. Then for each $j\in\{1,...,m\}$, take $w_j\in W\setminus O_j$. Let $L=\uparrow_X\{w_1,...,w_m\}\in Q(X)$. Then $L\subseteq W$ because $W$ is an open set. As a result, $L\in\Box W$. But $L\notin\bigcup_{j=1}^m\Box O_j$ due to the fact that $w_j\in L\setminus O_j$.
	
	To conclude, $\mathcal{U}\nll Q(X)$ and then $\mathcal{U}\ll\mathcal{V}$ if and only if $\mathcal{U}=\emptyset$, which is sufficient to show that $Q(X)$ is open well-filtered.
\end{proof}

\begin{theorem}
	The Hofmann-Lawson space $E$ is not an open well-filtered space but its Smyth power space is open well-filtered.
\end{theorem}

\begin{proof}
By Lemma 3.1, $Q(E)$ is open well-filtered. Now we prove $E$ is not open well-filtered.

For $n\ge 3$, let $N_n=(1-\frac{1}{n}, 1]$ and $b_n=\frac{1}{2}+\frac{1}{n}$. Define $H_n=N_n\times[0,b_n)$, $U_n=E\cap H_n$ and $G=E\cap(I\times[0,\frac{1}{2}))$. Clearly, $U_n$ and $G$ are open sets in $E$ for each $n$.

In order to prove that $\{U_n\}_{n\ge 3}$ is a $\ll$-filtered family, we define the map $\rho: \mathcal{O}(Y)\to \mathcal{O}(E)$, which sends each $G$ to the intersection $G\cap E$.

Claim 1: $\rho$ is an order-isomorphism between two open set lattices.

We only need to show that for $G, H\in\mathcal{O}(Y)$, $G\cap E\subseteq H\cap E$ implies $G\subseteq H$. Suppose $G\nsubseteq H$, take $(x,t)\in G\setminus H$. Consider $G_x=\{s\in J\mid (x,s)\in G\}$ and $H_x=\{s\in J\mid (x,s)\in H\}$. Note that the map $f$ from $J$ to $I\times J$ which sends each $s$ to $(x,s)$ is continuous and $G_x=f^{-1}(G), H_x=f^{-1}(H)$. Thus they are open sets in $J$ and we can take $g,h\in I$ such that $G_x=[0,g)$ and $H_x=[0,h)$. Since $(x,t)\in G\setminus H, t\in G_x$ but $t\notin H_x$, that is, $h\le t<g$. Therefore, $h<g$. Now consider $x$: If $x\in A$, take $s\in (h,g)\cap\mathbb{Q}$; If $x\in I\setminus A$, take $s\in (h,g)\cap(\mathbb{R}\setminus\mathbb{Q})$. No matter what the situation is, we can guarantee that $(x,s)\in E$. But $h<s<g$ implies $s\in G_x\setminus H_x$, whence $(x,s)\in G\setminus H$. To conclude, $G\cap E\nsubseteq H\cap E$, a contradiction.

Claim 2: $H_{n+1}\ll H_n$ for each $n\ge 3$.

This is obvious because $H_{n+1}\subseteq [1-\frac{1}{n+1}, 1]\times [0,b_{n+1}]\subseteq H_n$ and $[1-\frac{1}{n+1}, 1]\times [0,b_{n+1}]$ is compact in $Y$.

Note that $U_n=E\cap H_n$, then $U_{n+1}\ll U_n$ for each $n\ge 3$ by Claims 1 and 2, which implies $\{U_n\}_{n\ge 3}$ is a $\ll$-filtered family.

It is evident that $\bigcap_{n\ge 3} U_n=E\cap(\{1\}\times[0,\frac{1}{2}])=\{1\}\times([0,\frac{1}{2})\setminus\mathbb{Q})\subseteq G$, but $U_n\nsubseteq G$ for each $n$ because $A$ is dense and the "height" of $U_n$ is strictly bigger than $\frac{1}{2}$.
\end{proof}

\begin{lemma}
	Let $X$ be a locally compact space, $C$ a nonempty compact subset of $X$, and $U$ an open set with $C\subseteq U$. Then there exists $L\in Q(X)$ such that $C\subseteq int_X L\subseteq L\subseteq U$.
\end{lemma}

\begin{proof}
	For every $c\in C$, since $X$ is locally compact, there exists $H_c\in Q(X)$ such that $c\in int_X H_c\subseteq H_c\subseteq U$. Then $\{int_X H_c\}_{c\in C}$ is an open cover of $C$ and $C\subseteq\bigcup_{i=1}^n int_X H_{c_i}$ because $C$ is compact. Let $L=\bigcup_{i=1}^n H_{c_i}\in Q(X)$ and it is easy to see $C\subseteq int_X L\subseteq L\subseteq U$.
\end{proof}

\begin{theorem}
	A locally compact space $X$ is $PF$-well-filtered if and only if it is $PF$-sober.
\end{theorem}

\begin{proof}
	Every $PF$-sober space is $PF$-well-filtered.
	
	Conversely, given an irreducible set $\mathcal{A}\subseteq Q(X)$ and define $\mathcal{F}_{\mathcal{A}}=\{L\in Q(X) \mid \exists K\in\mathcal{A}, K\subseteq int_X L\}$.
	
	Since $\mathcal{A}\neq\emptyset$, take $K\in\mathcal{A}$. Then by Lemma 3.3, there exists $L\in Q(X)$ such that $K\subseteq int_X L\subseteq\ L\subseteq X$. So $L\in\mathcal{F}_{\mathcal{A}}$, which implies $\mathcal{F}_{\mathcal{A}}\neq\emptyset$.
	
	Given $L_1, L_2\in\mathcal{F}_{\mathcal{A}}$, then $\mathcal{A}\cap\Box(int_X L_i)\neq\emptyset$ for $i=1, 2$. Due to the irreducibility of $\mathcal{A}$, we have $\mathcal{A}\cap\Box(int_X L_1)\cap\Box(int_X L_2)\neq\emptyset$. Take $K_3$ in this intersection and by Lemma 3.3, there exists $L_3\in Q(X)$ such that $K_3\subseteq int_X L_3\subseteq L_3\subseteq int_X L_1\cap int_X L_2$. Thus $L_3\in\mathcal{F}_{\mathcal{A}}$ with $L_3\subseteq L_1\cap L_2$. Therefore, $\mathcal{F}_{\mathcal{A}}$ is filtered.
	
	Claim: $\bigcap\mathcal{F}_{\mathcal{A}}=\bigcap\mathcal{A}$.
	
	Assume $z\notin\bigcap\mathcal{A}$, then there exists $K\in\mathcal{A}$ such that $z\notin K$. Because $K$ is saturated, $z\notin G$ for some $G\in\mathcal{O}(X)$ and $K\subseteq G$. By Lemma 3.3, there exists $L\in Q(X)$ such that $K\subseteq int_X L\subseteq L\subseteq G$. Hence, $L\in\mathcal{F}_{\mathcal{A}}$ and $z\notin L$, which implies $z\notin\bigcap\mathcal{F}_{\mathcal{A}}$. To conclude, we have $\bigcap\mathcal{F}_{\mathcal{A}}\subseteq\bigcap\mathcal{A}$. The reverse inclusion is obvious.
	
	Therefore, if $\bigcap\mathcal{A}=\uparrow x\subseteq U\in\mathcal{O}(X)$, then $\bigcap\mathcal{F}_{\mathcal{A}}=\uparrow x\subseteq U$. Since $X$ is $PF$-well-filtered, there exists $L\in\mathcal{F}_{\mathcal{A}}$ satisfying $L\subseteq U$. Thus there exists $K\in\mathcal{A}$ such that $K\subseteq int_X L\subseteq L\subseteq U$. 
\end{proof}

Now we will provide characterizations of $PF$-well-filtered spaces and of $PF$-sober spaces separately.

\begin{lemma} (Topological Rudin's Lemma)\cite{heckmann2013quasicontinuous}
	Let $X$ be a topological space and $\mathcal{A}$ an irreducible subset of the
	Smyth power space $Q(X)$. Then every closed set $F\subseteq X$ that meets all members of $\mathcal{A}$ contains a minimal irreducible closed subset $C$ that still meets all members of $\mathcal{A}$.
\end{lemma}

\begin{proposition}
	Let $X$ be a $T_0$ space. Then $X$ is $PF$-well-filtered if and only if for each $C\in RD(X)$, $C=C^{ul}$.
\end{proposition}

\begin{proof}
	$(\Rightarrow)$ For any $C\in RD(X)$. Suppose $C^{ul}\nsubseteq C$, take $x\in C^{ul}\setminus C$. Let $U=X\setminus C$, then $U\in\mathcal{O}(X)$ with $x\in U$ and $C^u\subseteq\uparrow x$. Since $C$ is a Rudin set, take $\mathcal{K}$, which witnesses this. For each $K\in\mathcal{K}$, define $L_K$=$\uparrow(C\cap K)$ and $\mathcal{L}=\{L_K\mid K\in\mathcal{K}\}$. It is easy to see that $\mathcal{L}\subseteq Q(X)$ is filtered.
	
	Claim: $\bigcap_{K\in\mathcal{K}} L_K=C^{u}$.
	
	Assume $y\in\bigcap_{K\in\mathcal{K}} L_K$, then for each $K\in\mathcal{K}$, $y\ge c_K$ for some $c_K\in C\cap K$. Consider $B=C\cap\downarrow y$. Obviously, $B$ is closed and $B\subseteq C$ with $c_K\in B\cap K\neq\emptyset$. Thus $C\subseteq\downarrow y$ by the minimality of $C$, which implies $y\in C^u$. The reverse inclusion is obvious.
	
	Now define $\mathcal{M}=\{M_K\mid K\in\mathcal{K}\}$, where $M_K=L_K\cup\uparrow x$. It is evident that $\mathcal{M}\subseteq Q(X)$ is filtered and $\bigcap_{K\in\mathcal{K}} M_K=C^u\cup\uparrow x=\uparrow x\subseteq U$. So $M_{K_0}\subseteq U=X\setminus C$ for some $K_0\in\mathcal{K}$ by the $PF$-well-filteredness of $X$. But $\emptyset\neq C\cap K_0\subseteq L_{K_0}\subseteq M_{K_0}$, which implies $C\cap M_{K_0}\neq\emptyset$, a contradiction.
	
	Note that when $C^u=\emptyset$, we can also show that $\bigcap_{K\in\mathcal{K}} L_K=\emptyset$ and the conclusion still holds.
	
	$(\Leftarrow)$ Let $\mathcal{K}\subseteq Q(X)$ be filtered and $\bigcap\mathcal{K}=\uparrow x\subseteq U\in\mathcal{O}(X)$. Suppose for each $K\in\mathcal{K}$, $K\nsubseteq U$. Set $F=X\setminus U$, it is closed. Thus $K\cap F\neq\emptyset$ implies that there exists set $C$ satisfying the condition of Lemma 3.5. Consequently, $C\in RD(X)$ and $C^u\subseteq\bigcap\mathcal{K}=\uparrow x$ due to the fact that $C\cap K\neq\emptyset$. As a result, $x\in C^{ul}=C$, which implies that $x\in C\subseteq F=X\setminus U$, a contradiction.
\end{proof}

Similarly, one obtains the following characterization:

\begin{proposition}
	Let $X$ be a $T_0$ space. Then $X$ is $PF$-sober if and only if for each $C\in Irr_C(X)$, $C=C^{ul}$.
\end{proposition}

\begin{theorem}
	Let $X$ be a Rudin space. Then $X$ is $PF$-well-filtered if and only if $X$ is $PF$-sober.
\end{theorem}

\begin{proof}
	It is clear by Propositions 3.6 and 3.7.
\end{proof}

\begin{remark}
	Since every locally compact $T_0$ space is a Rudin space\cite{xu2020t0}, the result that $PF$-well-filteredness and $PF$-sobriety coincide for locally compact spaces can extend to the strictly larger class of the Rudin spaces.
\end{remark}

\section{On function spaces related to $PF$-well-filtered spaces and $PF$-sober spaces}

Given topological spaces $X$ and $Y$, let $TOP(X, Y)$ be the set of all continuous maps from $X$ to $Y$. In this section, we consider the Isbell topology on $TOP(X,Y)$ and show that if the function space consisting of all continuous maps from nonempty $X$ to $Y$ endowed with the Isbell topology is $PF$-well-filtered ($PF$-sober), then $Y$ is also $PF$-well-filtered ($PF$-sober). We also give a counterexample to show that for nonempty $T_0$ topological spaces $X$ and $Y$, if $Y$ is a $PF$-well-filtered ($PF$-sober) space, then the function space comprising all continuous maps from $X$ to $Y$ equipped with the Isbell topology may not be a $PF$-well-filtered ($PF$-sober) space even when $X$ is a core-compact space.

\begin{definition}\cite{gierz2003continuous}
	For two spaces $X$ and $Y$, let $H$ be a Scott open set of $\mathcal{O}(X)$, $V$ be an open set of $Y$ and set $$N(H\leftarrow V)=\{f\in TOP(X,Y): f^{-1}(V)\in H\}.$$
	The Isbell topology on $TOP(X,Y)$ is generated by $\{N(H\leftarrow V)\mid H\in\sigma(\mathcal{O}(X)), V\in\mathcal{O}(Y)\}$ as a subbase and this function space is denoted by $[X,Y]$.
\end{definition}

\begin{lemma}\cite{liu2021function}
	Let $X$ and $Y$ be two $T_0$ spaces. Given $y\in Y$, consider the mapping $\xi \colon Y \to [X,Y], \ y \mapsto \xi_y$, where $\xi_y(x)=y$ for all $x\in X$. Then $\xi$ is continuous.
\end{lemma}

\begin{lemma}\cite{li2023some}
	Let $X$ be a nonempty $T_0$ space. For each $x\in X$, define a function $F_x: [X,Y]\to Y$ by $F_x(f)=f(x)$ for any $f\in [X,Y]$. Then $F_x$ is continuous.
\end{lemma}

\begin{theorem}
	If $X$ is a nonempty $T_0$ space and $[X,Y]$ is a $PF$-well-filtered space, then $Y$ is a $PF$-well-filtered space.
\end{theorem}

\begin{proof}
	For any $A\in RD(Y)$, we have $\overline{\xi(A)}\in RD([X,Y])$ by Lemma 4.2. Since $[X,Y]$ is a $PF$-well-filtered space, $\overline{\xi(A)}=\overline{\xi(A)}^{ul}$. 
	
	Claim: For any $y\in A^{ul}$, $\xi_y\in\overline{\xi(A)}^{ul}$.
	
	Assume that $g\in\overline{\xi(A)}^u$, then $\xi(A)\subseteq\overline{\xi(A)}$, which implies $\xi(A)\subseteq\downarrow g$. For each $x\in X$, given $a\in A$, then $\xi_a\le g$. It follows that $a\le g(x)$. Thus $g(x)\in A^u$. So $y\le g(x)$. Therefore, $\xi_y\le g$. We conclude that $\xi_y\in\overline{\xi(A)}^{ul}$.
	
	Hence, $\xi_y\in\overline{\xi(A)}$. Given $x_0\in X$, note that $\xi(A)\subseteq F_{x_0}^{-1}(A)$ and $F_{x_0}$ is continuous, we have $\overline{\xi(A)}\subseteq F_{x_0}^{-1}(A)$; as a result, $y\in A$, which implies $A^{ul}=A$. Consequently, $Y$ is a $PF$-well-filtered space by Proposition 3.6.
\end{proof}

Similarly, with the help of Proposition 3.7, we have the following theorem.

\begin{theorem}
	If $X$ is a nonempty $T_0$ space and $[X,Y]$ is a $PF$-sober space, then $Y$ is a $PF$-sober space.
\end{theorem}

Now we give a counterexample to show that the reverse direction is not true even if $X$ is a core-compact space.

\begin{example}
	Take $X=\{a,b\}$ endowed with the discrete topology and $Y=\mathbb{N}$ with the Alexandroff topology under the usual order. It is clear that $X$ is core-compact and $Y$ is $PF$-sober.
	
	Claim 1: The product $Y^2$ is not $PF$-well-filtered.
	
	For each $n\ge 2$, let $K_n=\uparrow(n,1)\cup\uparrow(1,2)$. Obviously, it is a compact saturated set. Thus the filtered intersection $\bigcap_{n\ge 2} K_n=\uparrow(1,2)=\mathbb{N}\times\uparrow2\in\mathcal{O}(Y^2)$, but $K_n\nsubseteq\uparrow(1,2)$ for each $n\ge 2$.
	
	Claim 2: $[X,Y]\cong Y^2$.
	
	Define two maps $\phi$ and $\psi$ as follows:	
	\[
	\phi: [X,Y] \to Y \times Y, \qquad
	\psi: Y \times Y \to [X,Y],
	\]
	where
	\[
	\phi(f)=(f(a),f(b)),
	\]
	and
	\[
	\psi(y_1,y_2)\in [X,Y]
	\]
	satisfies
	\[
	\psi(y_1,y_2)(a)=y_1, \qquad
	\psi(y_1,y_2)(b)=y_2.
	\]
	
	It is evident that $\phi$ and $\psi$ are mutually inverse maps. We only need to prove that both maps are continuous.
	
	Note that $\sigma(\mathcal{O}(X))=\{\emptyset, \mathcal{O}(X), \{\{a\},\{a,b\}\}, \{\{b\},\{a,b\}\}, \{\{a\},\{b\},\{a,b\}\}, \{\{a, b\}\}\}.$ Denote $\mathcal{H}_a=\{\{a\},\{a,b\}\}$ and $\mathcal{H}_b=\{\{b\},\{a,b\}\}$. For each $U\in\mathcal{O}(Y)$, 
	
	\begin{align*}
		N(\mathcal{H}_a\leftarrow U) &= \{f: X\to Y\mid f^{-1}(U)\in\mathcal{H}_a\}\\
		&=\{f: X\to Y\mid f(a)\in U\}.
	\end{align*}
	Similarly, $N(\mathcal{H}_b\leftarrow U)=\{f: X\to Y\mid f(b)\in U\}$. Consider $\{U\times V\mid U, V\in\mathcal{O}(Y)\}$, which is a basis for the topology of $Y\times Y$. Note that 
	 
	 \begin{align*}
	 	\phi^{-1}(U\times V) &= \{f: X\to Y\mid f(a)\in U, f(b)\in V\}\\
	 	&= N(\mathcal{H}_a\leftarrow U)\cap N(\mathcal{H}_b\leftarrow V)\in \mathcal{O}([X,Y]).
	 \end{align*}
	 As a result, $\phi$ is continuous.
	
	For any $\mathcal{H}\in\sigma(\mathcal{O}(X))$, $U\in\mathcal{O}(Y)$, 
	\begin{align*}
		N(\mathcal{H}\leftarrow U) &= \{f: X\to Y\mid f^{-1}(U)\in\mathcal{H}\}\\
		&= \{f: X\to Y\mid \exists V\in\mathcal{H} \,\text{such that}\, V\subseteq f^{-1}(U)\}\\
		&=\{f: X\to Y\mid \exists V\in\mathcal{H}, \forall t\in V, f(t)\in U\}\\
		&=\bigcup_{V\in\mathcal{H}}\bigcap_{t\in V} \{f: X\to Y\mid f(t)\in U\}.
	\end{align*}
	It follows that $$\phi(N(\mathcal{H}\leftarrow U))=\bigcup_{V\in\mathcal{H}}\bigcap_{t\in V}\phi(\{f: X\to Y\mid f(t)\in U\}).$$
	Since $V\in\{\emptyset, \{a,b\}, \{a\}, \{b\}\}$,
	
	\[\bigcap_{t\in V}\phi(\{f: X\to Y\mid f(t)\in U\})=\begin{cases}
		Y\times Y, & V=\emptyset\\
		U\times Y, & V=\{a\}\\
		Y\times U, & V=\{b\}\\
		U\times U, & V=\{a,b\},
	\end{cases}\]
	which is open in $Y\times Y$. Thus $\phi(N(\mathcal{H}\leftarrow U))$ is the union of some open sets. Therefore, $$\psi^{-1}(N(\mathcal{H}\leftarrow U))=\phi(N(\mathcal{H}\leftarrow U))\in\mathcal{O}(Y\times Y).$$
	We conclude that $[X,Y]$ is not $PF$-well-filtered. Obviously, it is not a $PF$-sober space.
\end{example}

\section{The reflective hull}

In this section, we let $PFWF$ denote the category of all $PF$-well-filtered spaces with continuous maps and $PFSOB$ denote the category of all $PF$-sober spaces with continuous maps. We prove that the reflective hull of $PFSOB$ is $TOP_0$ and then the reflective hull of $PFWF$ is $TOP_0$ because $PFSOB\subseteq PFWF$.

\begin{definition}\cite{gierz2003continuous}
	A full subcategory $\mathcal{K}$ of $TOP_0$ is reflective if for each $T_0$ space $X$, there exists $KX\in\mathcal{K}$ (the $K$-completion for $X$) and a continuous mapping $\mu_X: X\to KX$ (the $K$-reflection for $X$) satisfying
	the universal property: for any continuous mapping $f: X\to Z$ to a space $Z\in\mathcal{K}$, there exists a unique continuous mapping $g: KX\to Z$ such that $g\circ\mu_X=f$.
\end{definition}

\begin{lemma}\cite{hou2023reflective}
	Let $\mathcal{C}$ be a subcategory of $TOP_0$ which contains a non-$T_1$ space. Then the reflective hull of $\mathcal{C}$ exists.
\end{lemma}

\begin{theorem}
	The reflective hull of $PFSOB$ in $TOP_0$ is $TOP_0$.
\end{theorem}

\begin{proof}
	It is easy to see that $PFSOB$ contains a non-$T_1$ space. Then the reflective hull of $PFSOB$ exists by Lemma 5.2. Let $I=[0,1]$ endowed with the Scott topology under the usual order. It is clear that $I$ is sober; consequently, it is $PF$-sober. For any $T_0$ space $X$, define $Y=I^{\mathcal{O}(X)}$. Then $Y$ is sober and $e: X\to Y$ sends each $x$ to $e(x)$, where 
	
	$$e(x)_U=\begin{cases}
		1, & x\in U\\
		0, & x\notin U
	\end{cases}$$
	for every $U\in\mathcal{O}(X)$ is a topological embedding since $e(U)=e(X)\cap\{z\in Y\mid z_U>\frac{1}{2}\}$.
	
	Claim: For each $p\in Y$, $A_p=Y\setminus \{p\}$ endowed with the relative topology from $Y$ is $PF$-sober.
	
	Take $C\in Irr_c(A_p)$, we need to prove that $C_{A_p}^{ul}=C$. Since the inclusion map from $A_p$ to $Y$ is continuous and $Y$ is sober, $\overline{C}^Y=\downarrow_Y q$ for some $q\in Y$. Note that $C$ is closed in $A_p$, thus $C=A_p\cap\downarrow_Y q$. Consider the following three situations:
	
	(1) If $q\neq p$, then $q\in A_p$ and $C=A_p\cap\downarrow_Y q=\downarrow_{A_p}q$ is a cut in $A_p$.
	
	(2) If $q=p=\top$, then $C=A_p\cap\downarrow_Y q=A_p$ is a cut in $A_p$.
	
	(3) If $q=p\neq\top$, take $U\in\mathcal{O}(X)$ such that $p_U<1$. Define
	
	 $$(q_n)_V=\begin{cases}
		p_U+\frac{1-p_U}{n+1}, & U=V\\
		p_V, & U\neq V
	\end{cases}$$
	for each $n\ge 1$. Since $q_n>p$, we have $q_n\in A_p$ and $\{q_n\mid n\ge 1\}\subseteq C_{A_p}^u$. It is clear that $\bigcap_{n\ge 1}\downarrow_Y q_n=\downarrow_Y p$. Consequently, $C_{A_p}^{ul}\subseteq\bigcap_{n\ge 1}\downarrow_Y q_n=\downarrow_Y p$. So $C_{A_p}^{ul}\subseteq\downarrow_Y p\cap A_p=\downarrow_Y q\cap A_p=C$, which implies that $C$ is a cut in $A_p$.
	
	Suppose $\mathcal{K}\subseteq TOP_0$ is a reflective category with $PFSOB\subseteq\mathcal{K}$. Given a $T_0$ space $X$ and let $\eta: X\to KX$ be the $K$-reflection. Since $Y\in PFSOB$, there exists a unique continuous map $f: KX\to Y$ such that $f\circ\eta=e$. For each $p\in Y\setminus e(X)$, the map $e_p: X\to A_p$ is a topological embedding, and $A_p\in\mathcal{K}$ by the Claim above. Thus there exists a unique continuous map $f_p: KX\to A_p$ such that $f_p\circ\eta=e_p$. Let $i_p$ be the inclusion from $A_p$ to $Y$. Then $i_p\circ f_p=f$ by the $K$-reflection $\eta$ and the object $Y$. Therefore, $f(KX)\subseteq f_p(KX)\subseteq A_p$ for each $p\in Y\setminus e(X)$, which implies that $f(KX)\subseteq\bigcap_{p\in Y\setminus e(X)} Y\setminus\{p\}=e(X)$. Let $e^{-1}: e(X)\to X$ denote the inverse homeomorphism. Then $r=e^{-1}\circ f: KX\to X$ is continuous. Thus $r\circ\eta=e^{-1}\circ f\circ\eta=e^{-1}\circ e=id_X$. 	Note that $(\eta\circ r)\circ\eta=\eta\circ (r\circ\eta)=\eta$. So $\eta\circ r=id_{KX}$ by the $K$-reflection $\eta$ and the object $KX$. We conclude that $X$ is homeomorphic to $KX$ with the continuous map $r$ and $\eta$. Therefore, $X\in\mathcal{K}$, which implies $\mathcal{K}=TOP_0$.
	
\end{proof}

\begin{corollary}
	The reflective hull of $PFWF$ in $TOP_0$ is $TOP_0$.
\end{corollary}

\begin{proof}
	This is an immediate consequence of Theorem 5.3.
\end{proof}

\section*{Declarations}

Conflicts of interest\quad The authors declare no conflicts of interest.

Funding: This work is supported by the National Natural Science Foundation of China (Grant Nos.12231007 and 12571498).


\begin{thebibliography}{14}
		\expandafter\ifx\csname natexlab\endcsname\relax\def\natexlab#1{#1}\fi
		\providecommand{\url}[1]{\texttt{#1}}
		\providecommand{\href}[2]{#2}
		\providecommand{\path}[1]{#1}
		\providecommand{\DOIprefix}{doi:}
		\providecommand{\ArXivprefix}{arXiv:}
		\providecommand{\URLprefix}{URL: }
		\providecommand{\Pubmedprefix}{pmid:}
		\providecommand{\doi}[1]{\href{http://dx.doi.org/#1}{\path{#1}}}
		\providecommand{\Pubmed}[1]{\href{pmid:#1}{\path{#1}}}
		\providecommand{\bibinfo}[2]{#2}
		\ifx\xfnm\relax \def\xfnm[#1]{\unskip,\space#1}\fi
		\bibitem[{Gierz et~al.(2003)Gierz, Hofmann, Keimel, Lawson, Mislove and
			Scott}]{gierz2003continuous}
		\bibinfo{author}{Gierz, G.}, \bibinfo{author}{Hofmann, K.H.},
		\bibinfo{author}{Keimel, K.}, \bibinfo{author}{Lawson, J.D.},
		\bibinfo{author}{Mislove, M.}, \bibinfo{author}{Scott, D.S.},
		\bibinfo{year}{2003}.
		\newblock \bibinfo{title}{Continuous lattices and domains}.
		Volume~\bibinfo{volume}{93}.
		\newblock \bibinfo{publisher}{Cambridge University Press}.
		\bibitem[{Goubault-Larrecq(2013)}]{goubault2013non}
		\bibinfo{author}{Goubault-Larrecq, J.}, \bibinfo{year}{2013}.
		\newblock \bibinfo{title}{Non-Hausdorff topology and domain theory: Selected
			topics in point-set topology}. Volume~\bibinfo{volume}{22}.
		\newblock \bibinfo{publisher}{Cambridge University Press}.
		\bibitem[{Heckmann and Keimel(2013)}]{heckmann2013quasicontinuous}
		\bibinfo{author}{Heckmann, R.}, \bibinfo{author}{Keimel, K.},
		\bibinfo{year}{2013}.
		\newblock \bibinfo{title}{Quasicontinuous domains and the Smyth powerdomain}.
		\newblock \bibinfo{journal}{Electronic Notes in Theoretical Computer Science}
		\bibinfo{volume}{298}, \bibinfo{pages}{215--232}.
		\bibitem[{Hofmann and Lawson(1978)}]{hofmann1978spectral}
		\bibinfo{author}{Hofmann, K.H.}, \bibinfo{author}{Lawson, J.D.},
		\bibinfo{year}{1978}.
		\newblock \bibinfo{title}{The spectral theory of distributive continuous
			lattices}.
		\newblock \bibinfo{journal}{Transactions of the American Mathematical Society}
		\bibinfo{volume}{246}, \bibinfo{pages}{285--310}.
		\bibitem[{Hou et~al.(2023)Hou, Li, Miao and Zhao}]{hou2023reflective}
		\bibinfo{author}{Hou, H.}, \bibinfo{author}{Li, Q.}, \bibinfo{author}{Miao,
			H.}, \bibinfo{author}{Zhao, D.}, \bibinfo{year}{2023}.
		\newblock \bibinfo{title}{The reflective hulls of some subcategories in the
			category of all $T_0$ spaces}.
		\newblock \bibinfo{journal}{Houston Journal of Mathematics}
		\bibinfo{volume}{49}, \bibinfo{pages}{381--395}.
		\bibitem[{Li et~al.(2023)Li, Jin, Miao and Chen}]{li2023some}
		\bibinfo{author}{Li, Q.}, \bibinfo{author}{Jin, M.}, \bibinfo{author}{Miao,
			H.}, \bibinfo{author}{Chen, S.}, \bibinfo{year}{2023}.
		\newblock \bibinfo{title}{On some results related to sober spaces}.
		\newblock \bibinfo{journal}{Acta Mathematica Scientia} \bibinfo{volume}{43},
		\bibinfo{pages}{1477--1490}.
		\bibitem[{Li et~al.(2020)Li, Yuan and Zhao}]{li2020unified}
		\bibinfo{author}{Li, Q.}, \bibinfo{author}{Yuan, Z.}, \bibinfo{author}{Zhao,
			D.}, \bibinfo{year}{2020}.
		\newblock \bibinfo{title}{A unified approach to some non-hausdorff topological
			properties}.
		\newblock \bibinfo{journal}{Mathematical Structures in Computer Science}
		\bibinfo{volume}{30}, \bibinfo{pages}{997--1010}.
		\bibitem[{Liu et~al.(2021)Liu, Li and Ho}]{liu2021function}
		\bibinfo{author}{Liu, B.}, \bibinfo{author}{Li, Q.}, \bibinfo{author}{Ho,
			W.K.}, \bibinfo{year}{2021}.
		\newblock \bibinfo{title}{On function spaces related to d-spaces}.
		\newblock \bibinfo{journal}{Topology and its Applications}
		\bibinfo{volume}{300}, \bibinfo{pages}{107757}.
		\bibitem[{Lu and Li(2017)}]{lu2017weak}
		\bibinfo{author}{Lu, C.}, \bibinfo{author}{Li, Q.}, \bibinfo{year}{2017}.
		\newblock \bibinfo{title}{Weak well-filtered spaces and coherence}.
		\newblock \bibinfo{journal}{Topology and its Applications}
		\bibinfo{volume}{230}, \bibinfo{pages}{373--380}.
		\bibitem[{Shen et~al.(2019)Shen, Xi, Xu and Zhao}]{shen2019well}
		\bibinfo{author}{Shen, C.}, \bibinfo{author}{Xi, X.}, \bibinfo{author}{Xu, X.},
		\bibinfo{author}{Zhao, D.}, \bibinfo{year}{2019}.
		\newblock \bibinfo{title}{On well-filtered reflections of $T_0$ spaces}.
		\newblock \bibinfo{journal}{Topology and its Applications}
		\bibinfo{volume}{267}, \bibinfo{pages}{106869}.
		\bibitem[{Shen et~al.(2020)Shen, Xi, Xu and Zhao}]{shen2020open}
		\bibinfo{author}{Shen, C.}, \bibinfo{author}{Xi, X.}, \bibinfo{author}{Xu, X.},
		\bibinfo{author}{Zhao, D.}, \bibinfo{year}{2020}.
		\newblock \bibinfo{title}{On open well-filtered spaces}.
		\newblock \bibinfo{journal}{Logical Methods in Computer Science}
		\bibinfo{volume}{16}.
		\bibitem[{Shen et~al.(2023)Shen, Xi and Zhao}]{shen2023further}
		\bibinfo{author}{Shen, C.}, \bibinfo{author}{Xi, X.}, \bibinfo{author}{Zhao,
			D.}, \bibinfo{year}{2023}.
		\newblock \bibinfo{title}{Further studies on open well-filtered spaces}.
		\newblock \bibinfo{journal}{Electronic Notes in Theoretical Informatics and
			Computer Science} \bibinfo{volume}{2}.
		\bibitem[{Xu(2025)}]{xu2025strongly}
		\bibinfo{author}{Xu, X.}, \bibinfo{year}{2025}.
		\newblock \bibinfo{title}{Strongly well-filtered spaces and strong d-spaces}.
		\newblock \bibinfo{journal}{Topology and its Applications} ,
		\bibinfo{pages}{109654}.
		\bibitem[{Xu et~al.(2020)Xu, Shen, Xi and Zhao}]{xu2020t0}
		\bibinfo{author}{Xu, X.}, \bibinfo{author}{Shen, C.}, \bibinfo{author}{Xi, X.},
		\bibinfo{author}{Zhao, D.}, \bibinfo{year}{2020}.
		\newblock \bibinfo{title}{On $T_0$ spaces determined by well-filtered spaces}.
		\newblock \bibinfo{journal}{Topology and its Applications}
		\bibinfo{volume}{282}, \bibinfo{pages}{107323}.	
\end{thebibliography}

\end{document}